\documentclass[12pt,a4paper]{article}
\usepackage[T2A]{fontenc}
\usepackage[cp1251]{inputenc}
\usepackage[english]{babel}
\usepackage{epsfig,amssymb,epic,eepic,color, graphicx}
\usepackage{fancyhdr}
\usepackage{amsthm}
\usepackage{amsmath}
\usepackage{indentfirst}
\usepackage{float}
\usepackage{graphicx}
\graphicspath{}
\DeclareGraphicsExtensions{.pdf,.png,.jpg}
\usepackage{setspace}
\usepackage{float}
\usepackage{fancyhdr}
\usepackage{geometry}
\newtheorem{teo}{Theorem}
\newtheorem{stat}{Statement}[section]

\newtheorem{lem}{Lemma}[section]

\renewenvironment{abstract}{\textbf{Аbstract.} }{\medskip}
\begin{document} 

\begin{center}{\bf On homeomorphisms of three-dimensional manifolds with pseudo-Anosov attractors and repellers}
 
Grines V.Z., Pochinka O.V., Chilina E.E.
\end{center}

\begin{abstract}
The present paper is devoted to a study of orientation-preserving homeomorphisms on three-dimensional manifolds with a non-wandering set consisting of a finite number of surface attractors and repellers. The main results of the paper relate to a class of homeomorphisms for which the restriction of the map to a connected component of the non-wandering set is topologically conjugate to an orientation-preserving pseudo-Anosov homeomorphism. The ambient $\Omega$-conjugacy of a homeomorphism from the class with a locally  direct product of a pseudo-Anosov homeomorphism and a rough transformation of the circle is proved. In addition, we prove that the centralizer of a pseudo-Anosov homeomorphisms consists of only  pseudo-Anosov and periodic maps.
\end{abstract}


\section{Introduction} 
In \cite{GrinesLevchenkoMedvedevPochinka,GrinesLevchenkoMedvedevPochinkaNonlin} the dynamics of three-dimensional $A$-diffeomorphisms was studied under the assumption that their non-wandering set consists of surface two-dimensional basic sets. It is proved that diffeomorphisms of this class are ambiently $\Omega$-conjugate to locally direct products of an Anosov diffeomorphism of a two-dimensional torus and a rough transformation of a circle. 
This work is a generalization of these results to a wider class  $\mathcal G$ of maps, which we define as follows. 

The set $\mathcal G$ consists of orientation-preserving homeomorphisms $f$ of a closed orientable topological 3-manifold $M^3$ with the non-wandering set $NW(f)$ consisting of a finite number of connected components $B_0,\dots,B_{m -1}$  satisfying for any $i\in\{0,\dots,m-1\}$ the following conditions:
   \begin{enumerate}
   \item $B_i$ is a cylindrical\footnote{A subspace $X$ of a topological space $Y$ is called {\it a cylindrical embedding into $Y$ of a topological space $\bar X$} if there is a homeomorphism onto the image $h:\bar X\times[-1,1]\to Y$ such that $X=h(\bar X\times\{0\})$.} embedding of a closed orientable surface of genus greater than $1$;

   \item there is a natural number $k_i$ such that $f^{k_i}( B_i)= B_i$, $f^{\tilde k_i}( B_i)\neq B_i$ for any natural number $\tilde k_i< k_i$ and the restriction of the map $f^{k_i}|_{ B_i}$ is topologically conjugate to an orientation-preserving pseudo-Anosov homeomorphism;

   \item $B_i$ is either an attractor\footnote{An invariant set $B$ of a homeomorphism $f$ is called an {\it attractor} if there is a closed neighborhood $U$ of the set $B$ such that $f(U)\subset int\ ; U$, $\underset{j\geq 0}{\bigcap}f^j(U)=B$. The attractor for the homeomorphism $f^{-1}$ is called the {\it repeller} of the homeomorphism $f$.}  or a repeller for the homeomorphism $f^{k_i}$.
   
  \end{enumerate}

  The simplest representatives of the class $\mathcal G$ are homeomorphisms of the set $\Phi$ which are constructed as follows.

Represent the circle as a subset of the complex plane $\mathbb S^1=\{e^{i2\pi\theta}|0\leq \theta <1\}$ and define a covering $p\colon \mathbb R \to\mathbb S^1$ so that $p(r)=s$, where $s=e^{i2\pi r}$.

Consider sets of numbers $n,k,l$ such that $n,k\in\mathbb N$, $l\in\mathbb Z$, where $l=0$ if $k=1$, and $l \in\{1,\dots,k-1\}$   is coprime to $k$ if $k>1$. For each set $n,k,l$ we define a diffeomorphism $\bar\varphi_{n,k,l}:\mathbb R\to\mathbb R$ by the formula \begin{equation*} \bar\varphi_{n,k, l}(r)=r+\frac{1}{4\pi nk} sin(2\pi nkr)+\frac{l}{k}. \end{equation*}

Since $\bar\varphi_{n,k,l}(r)+1=\bar\varphi_{n,k,l}(r+1)$, it follows that the diffeomorphism $\bar\varphi_{n,k, l}$ is the lift of the circle map $\varphi_{n,k,l}(s)=p(\bar\varphi_{n,k,l}(p^{-1}(s)))$, where $ p^{-1}(s)$ is the preimage of the point $s\in\mathbb S^1$ (see Statement \ref{cyclic}).

Denote by $S_g$  a closed orientable surface of genus $g>1$ and by $Z(P)$ the centralizer $Z(P)=\{J\colon S_g\to S_g|\,PJ=JP\}$ of a homeomorphism $P\colon S_g\to S_g$.

  Let us denote by $\mathcal P$  the set of all pseudo-Anosov homeomorphisms on the surface $S_g$.

  \begin{teo}\label{prJ}
  A homeomorphism $J\in Z(P)$, where $P\in\mathcal P$, is either   pseudo-Anosov or  periodic\footnote{A homeomorphism $f$ is called periodic if there exists $m\in\mathbb N$ such that $f^m=id$.}. \end{teo}

  Consider orientation-preserving homeomorphisms $P\in\mathcal P$ and $J\in Z(P)$ such that the map $J^lP^k$ is a pseudo-Anosov homeomorphism. Let us represent the manifold $M_J$ as the quotient space of the manifold $S_g\times\mathbb R$ by the action of the group $\Gamma=\{\gamma^i, i\in\mathbb Z\}$ of degrees of homeomorphism $\gamma\colon S_g\times\mathbb R \to S_g\times\mathbb R$, given by the formula $\gamma(z,r)=(J(z),r-1)$, with natural projection $p_{_{J}} \colon S_g\times \mathbb R \to M_{{J}}$.

Define the map $\bar\varphi_{P,J,n,k,l}\colon S_g\times\mathbb R \to S_g\times\mathbb R$ by the formula \begin{equation*} \bar\varphi_{P,J ,n,k,l}(z,r)=(P(z),\bar\varphi_{n,k,l}(r)). \end{equation*}

It is readily verified that $\bar\varphi_{P,J,n,k,l}\gamma=\gamma\bar\varphi_{P,J,n,k,l}$. Then the orientation-preserving homeomorphism $\varphi_{P,J,n,k,l}:M_J\to M_J$ is correctly defined (see Statement \ref{cyclic}) and given by the formula \begin{equation*} \varphi_{P, J,n,k,l}(w)=p_{_J}(\bar \varphi_{P,J,n,k,l}(p_{_J}^{-1}(w))), \end{equation*} where $w\in M_J$ and $p_{J}^{-1}(w)$ is the preimage of the point $w\in M_J$. We call homeomorphisms of the form $\varphi_{P,J,n,k,l}$ {\it model maps}. Denote by $\Phi$   the set of all model maps.

\begin{teo}\label{modelmaps}
 Any homeomorphism from the class $\Phi$ belongs to the class $\mathcal G$.
  \end{teo}

   \begin{teo}\label{omega} Any homeomorphism from the class $\mathcal G$ is ambiently $\Omega$-conjugate\footnote{Recall that   homeomorphisms $f_1\colon X\to X$ and $f_2\colon Y\to Y$ of topological manifolds $X$ and $Y$ are called {\it ambiently $\Omega$-conjugated} if there is a homeomorphism $h\colon X\to Y$ such that $h(NW(f_1))=NW (f_2)$ and $hf_1|_{NW(f_1)}=f_2h|_{NW(f_1)}$.} to a homeomorphism from the class $\Phi$. \end{teo}

   \section{Main definitions and auxiliary statements}
 
  \subsection{Pseudo-Anosov homeomorphisms}
 
   Let $M^n$ be a topological manifold of dimension $n$.
  
   Family $\mathcal F=\{ L_{\alpha}; \alpha\in A\}$ of path-connected subsets in $M^n$ is called a {\it $k$-dimensional foliation} if it satisfies the following three conditions:
   \begin{itemize}
\item $L_{\alpha}\cap L_{\beta}=\emptyset$ for any $\alpha,\beta\in A$ such that $\alpha\neq\beta$;
\item $\underset{\alpha\in A}{\bigcup} L_{\alpha}=M^n$;
\item for any point $p\in M^n$ there is a local map $(U,\varphi)$, $p\in U$, so that if $U\cap L_{\alpha}\neq\emptyset $, $\alpha\in A$, then the path-connected components of the set $\varphi(U\cap L_{\alpha})$ have the form $\{(x_1,x_2,\dots,x_n)\in\varphi(U )$; $x_{k+1}=c_{k+1}, x_{k+2}=c_{k+2}, \dots, x_n=c_n\}$, where the numbers $c_{k+1}, c_{k+2},\dots,c_{n}$ are constant on the linearly connected components.
\end{itemize}

{\it A foliation $\mathcal F$ with a set of singularities $S$ of  $M^n$} is a family of path-connected subsets of  $M^n$ such that the family of sets $\mathcal F\setminus S $ is a foliation of  $ M^n\setminus F$.

  Let $q\in\mathbb N$. {\it The foliation $W_q$ on $\mathbb C$ with the standard saddle singularity at the point $O$ and $q$ separatrices} is a family of path-connected subsets in $\mathbb C$ such that $W_q\setminus O$ is a foliation on $\mathbb C\setminus O$ and   $Im\; z^{\frac{q}{2}}=const$ on leaves of $W_q\setminus O$. Rays $l_1,\dots,l_q\in W_q$ satisfying equality $Im\; z^{\frac{q}{2}}=0$ are called separatrices of the point $O$.

 \begin{figure}[!ht]
\includegraphics[scale=1.5]{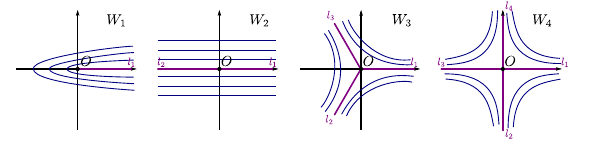}
\caption{The foliation $W_q$ on $\mathbb C$ with the standard
saddle singularity at the point $O$   and $q$ separatrices   for $q=1,2,3,4$.}
\label{saddles}
\end{figure}

   A one-dimensional foliation $\mathcal F$ on $M^2$ is called a {\it foliation with saddle singularities} if the set $S$ of singularities of the foliation $\mathcal F$ consists of a finite number of
points $s_1,\dots,s_c$ and for any point $s_i$ ($i\in\{1,\dots,c\}$) there is a neighborhood $U_{i}\subset M^2$,
a homeomorphism $\psi_{i}\colon U_{i}\to\mathbb C$ and a number $q_i\in\mathbb N$ such that $\psi_{i}(s_i)=O$
and $\psi_{i}(\mathcal F\cap U_{i})=W_{q_i}\setminus\{O\}$. The leaf containing the curve $\psi_i^{-1}(l_j)$, $j\in\{1,\dots,q_i\}$, is called the separatrix of the point $s_i$. The point $s_i$ is called {\it a saddle singularity
with $q_i$ separatrices}.

{\it The transversal measure $\mu$} for a foliation $\mathcal F$ with saddle singularities  on $M^2$ associates with each arc $\alpha$ transversal to $\mathcal F$ a non-negative Borel measure $\mu|_{ \alpha}$ with the following properties:
  \begin{enumerate}
  
   \item if $\beta$ is a subarc of the arc $\alpha$, then $\mu|_{\beta}$ is a restriction of the measure $\mu|_{\alpha}$;
  
   \item if $\alpha_0$ and $\alpha_1$ are two arcs transversal to $\mathcal F$ and connected by   a homotopy $\alpha\colon [0,1]\times [0,1]\to M^2 $ such that $\alpha( [0,1]\times\{0\})=\alpha_0$, $\alpha( [0,1]\times\{1\})=\alpha_1$ and $\alpha (\{t\}\times [0,1])$ for any $t\in [0,1]$ is contained in a leaf of $\mathcal F$ (see Fig. \ref{homotopic curves}), then $\mu|_{\alpha_0}=\mu|_{\alpha_1}$.
\end{enumerate}
\begin{figure}[!ht]
\includegraphics[scale=1.5]{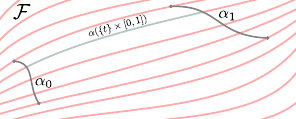}
\caption{Curves $\alpha_0$ and $\alpha_1$ are connected by homotopy $\alpha$.}
\label{homotopic curves}
\end{figure}

An orientation-preserving homeomorphism $P\colon S_g\to S_g$ of a closed orientable surface of genus $g>1$ is called a {\it pseudo-Anosov map ($pA$-homeomorphism)} with {\it dilatation} $\lambda>1$ if on surface $S_g$ there is a pair of $P$-invariant transversal foliations $\mathcal F^s_P$, $\mathcal F^u_P$ with a set of saddle singularities $S$ and transversal measures $ \mu_s$, $ \mu_u$ such that :
\begin{itemize}
\item each saddle singularity from $S$ has at least three separatrices;
\item $ \mu_s(P(\alpha))=\lambda \mu_s(\alpha)$ ($ \mu_u(P(\alpha))=\lambda^{-1} \mu_u(\alpha)$) for any arc $\alpha$ transversal to $\mathcal F^s_P$ ($\mathcal F^u_P$).
\end{itemize}

Let $P\colon S_g\to S_g$ be a pseudo-Anosov homeomorphism.
 Define the stable (unstable) manifold  $W^s(x)=\{y\in M^3:d(P^n(x),P^n(y))\to 0, n\to +\infty\}$  ($W^u(x)=\{y\in M^3:d(P^n(x),P^n(y))\to 0, n\to -\infty\}$)  of  $x\in S_g$, where $d$ is a metric on $S_g$. Note that the stable (unstable) manifold of the point $x\notin S$ is a leaf of the foliation $\mathcal F^s_P$ ($\mathcal F^u_P$) and  a stable (unstable) manifold of the point $x\in S$ is the union of a finite number of separatrices belonging to the foliation $\mathcal F^s_P$ ($\mathcal F^u_P$) and the point $x$.


{\it A rectangle} is a subset $\Pi\subset S_g$ that is the image of a continuous map $\upsilon$ of the square $[0,1]\times[0,1]$ into $S_g$ with the following properties: $\upsilon$ is one-to-one on the interior of the square and maps segments of its horizontal partition into arcs of leaves $\mathcal F^s_P$, and segments of its vertical partition into arcs of leaves $\mathcal F^u_P$. Denote by $\dot \Pi$  the image of the interior of the square. We will call the images of the horizontal and vertical sides   contracting and stretching sides of the rectangle $\Pi$.

{\it A Markov partition for a pseudo-Anosov homeomorphism $P$} is a finite family of rectangles $\tilde \Pi=\{\Pi_1,\dots,\Pi_n\}$ for which the following conditions are satisfied:
\begin{itemize}
\item $\underset{i}{\bigcup}\Pi_i=S_g$; $\dot \Pi_i\cap \dot\Pi_j=\emptyset$ for $i\neq j$;

\item let $\partial^s \tilde \Pi$ ($\partial^u \tilde \Pi$) be the union of all contracting (stretching) sides of rectangles $\Pi_1,\dots,\Pi_n$, then $P( \partial^s\tilde\Pi)\subset \partial^s\tilde\Pi$; $P(\partial^u\tilde\Pi)\supset \partial^u\tilde\Pi$.
\end{itemize}

\begin{stat}[\cite{Laudenbach}, Proposition 10.17]\label{Markov}
A pseudo-Anosov homeomorphism has a Markov partition.
  \end{stat}

  A foliation $\mathcal F$ is called {\it  uniquely ergodic} if there
exists a single $\mathcal F$-invariant measure up to multiplication by a scalar.

\begin{stat}[\cite{Laudenbach},  Theorem 12.1]\label{conjugacypA}
The foliations $\mathcal F^s_P$ and $\mathcal F^u_P$ of the pseudo-Anosov homeomorphism $P$ are uniquely ergodic.
\end{stat}

\begin{stat}[\cite{Laudenbach}, Theorem 12.5]\label{conjugacypA}
Two homotopic
pseudo-Anosov diffeomorphisms are conjugate by a diffeomorphism
isotopic to the identity.
\end{stat}

\begin{stat}[\cite{Zhirov}, Lemma 3.1]\label{pA2}
A homeomorphism that is topologically conjugate to a pseudo-Anosov homeomorphism is also pseudo-Anosov.
\end{stat}
 
\begin{stat}[\cite{Zhirov},  Theorem 3.2]\label{dense}
The set of periodic points of a pseudo-Anosov homeomorphism is dense everywhere on the surface.
 \end{stat}

  \begin{stat}[\cite{Zhirov}, Note 3.6]\label{denseleaf}
Every leaf of  foliations $\mathcal F^s_P$ and $\mathcal F^u_P$ of the pseudo-Anosov homeomorphism $P$ is everywhere dense on $S_g$.
 \end{stat}

  \subsection{Group action on a topological space}

Let us recall some facts related to the action of a group on a topological space (for more details, see \cite{GrPo}).

For a continuous mapping $h\colon X\to Y$ of a topological space $X$ into a topological space $Y$, denote by  $h^{-1}(V)$ the preimage of the set $V\subset Y$, that is, $h ^{-1}(V)=\{x\in X|h(x)\in V\}$.

Let the action of a group $G$ be free and discontinuous on a Hausdorff space $X$ and
let the orbits space $X/G$ be connected.   The definition of the projection $p_{ X/G}\colon X\to X/G$
implies that $p^{-1}_{X/G}(x)$ 
  is an orbit of some point $\bar x\in p^{-1}_{ X/G}(x)$. Let $c$ be a path in $X/G$ 
for which $c(0)=c(1)=x$. The monodromy theorem implies that there is the unique
path $\bar c$  in  $X$  starting from  $\bar x$ ($\bar c(0)=\bar x)$ which is a lift of the path$c$. Therefore, there
is an element  $g\in G$  for which $\bar c(1)=g(\bar x)$. Hence, the map $\eta_{X/G,\bar x}\colon \pi_1(X/ G,x)\to G$
defined by $\eta_{X/G,\bar x}([c])=g$ is well defined, i.e. it is independent of the choice of the
path in the class $[c]$.

\begin{stat}[\cite{GrPo},  Statement 10.32]\label{ep}  The map $\eta_{X/G,\bar x}\colon \pi_1(X/G,x)\to G $  is a nontrivial homomorphism.
It is called the homomorphism induced by the cover $p_{X/G}\colon X\to X/G$.
  \end{stat}

 Let $G$ be an abelian group and let $\bar c'$ be the lift of a path $c\in\pi_1(X/G,x)$ starting
from a point $\bar x'=\bar c'(0) $ distinct from the point $\bar x$ and let $g'(\bar x')=\bar c'(1)$. Since there
is the unique element $ g''\in G$  for which  $g''(\bar x)=\bar x'$ the monodromy theorem implies
 $g''(\bar c)=\bar c'$. Then $g''g=g'g''$ and, therefore, $g'=g$. Thus $\eta_{X/G,\bar x}=\eta_{X/G,\bar x'}$  and from
now on we omit the index $\bar x$ in the notation of the epimorphism $\eta_{X/G,\bar x}$ and we write
$\eta_{X/G}$ if $G$ is an abelian group.

\begin{stat}[\cite{GrPo}, Statement 10.35]\label{cyclic} Let cyclic groups $G$, $G'$  act freely and discontinuously on  $G$, $G'$-
space $X$ and let $g$, $g'$ be their respective generators. Then 
\begin{enumerate}
\item if $\bar h\colon X\to X$  is a homeomorphism for which  $\bar h(g(\bar x))=g'(\bar h(\bar x))$  for every $\bar x\in X$ then
the map $h\colon X/G\to X/{G'}$ defined by $h=p_{X/{G'}}(\bar h(p_{ X/G }^{-1}( x)))$ is a homeomorphism and
$\eta_{X/G}=\eta_{X/G'}h_*$;

\item if $h\colon X/G\to X/{G'}$ is a homeomorphism for which $\eta_{X/G}=\eta_{X/G'}h_*$ then there is
the unique homeomorphism $\bar h\colon X\to X$ which is a lift of $h$  and such that  $\bar h(g(\bar x))=g'(\bar h(\bar x))$, $\bar h(\bar x)=\bar x'$ for $\bar x\in X$ and $\bar x'\in p^{-1}_{ X/G'}( x')$, where $x'=h(p_{ X/G}(\bar x))$.
  \end{enumerate}
  \end{stat}

\section{On the centralizer of a pseudo-anosov map}\label{proofprJ}

In this section we prove that a homeomorphism $J\in Z(P)$, where $P\in\mathcal P$, is either a pseudo-Anosov homeomorphism or a periodic homeomorphism.
   
\begin{proof}
 Let $P\in\mathcal P$ and $J\in Z(P)$. Since $P=JPJ^{-1}$, it follows that  $J$ maps stable manifolds of   $P$ into stable ones, and unstable ones into unstable ones. Therefore, $J(\mathcal F^s_P)=\mathcal F^s_P$ and $J(\mathcal F^u_P)=\mathcal F^u_P$. The foliations $\mathcal F^s_P$, $\mathcal F^u_P$ have transversal measures $\mu_s$, $\mu_u$. Let us define for the foliation $\mathcal F^s_P$ ($\mathcal F^u_P$) a transversal measure $\tilde \mu_s(\alpha_s)=\mu_s(J(\alpha_s))$ ($\tilde \mu_u(\alpha_u )=\mu_u(J(\alpha_u))$), where $\alpha_s$ ($\alpha_u$) is the arc transversal to the foliation $\mathcal F^s_P$ ($\mathcal F^u_P$). Since  foliations $\mathcal F^s_P$, $\mathcal F^u_P$ are uniquely ergodic (Proposition \ref{conjugacypA}),  there exist numbers $\nu_s,\nu_u\in\mathbb R_+$ such that $ \tilde\mu_s=\nu_s\mu_s$ and $\tilde\mu_u=\nu_u\mu_u$. Thus, $\mu_s(J(\alpha_s))=\nu_s\mu_s(\alpha_s)$, $\mu_u(J(\alpha_u))=\nu_u\mu_u(\alpha_u)$ for arc $\alpha_s$ transversal to $\mathcal F^s_P$ and the arc $\alpha_u$ transversal to $\mathcal F^u_P$.

    Since the pseudo-Anosov homeomorphism $P$ has a Markov partition (see Statement \ref{Markov}) consisting of $n$ rectangles $\Pi_1,\dots,\Pi_n$, it follows that on each rectangle $\Pi_i$ ($i\in\{1,\dots,n\}$) the measure $\mu_s\otimes\mu_u$ is defined by the formula $\mu_s\otimes\mu_u(\Pi_i)=\mu_s(\alpha_{s,i})\mu_u( \alpha_{u,i})=\mu_i$, where $\alpha_{s,i}$ is the stretching side of the rectangle $\Pi_i$ and $\alpha_{u,i}$ is the contracting side. Since the foliations $\mathcal F^s_P$, $\mathcal F^u_P$ are invariant under $J$, it follows that the set $J( \Pi_i)$ ($i\in\{1,\dots,n\}$) is also a rectangle with measure $\mu_s\otimes\mu_u(J(\Pi_i))=\mu_s(J(\alpha_{s,i}))\mu_u(J(\alpha_{u,i}))=\nu_s\nu_u\mu_i$. Thus, $\mu_s\otimes\mu_u(S_g)=\mu_s\otimes\mu_u(\underset{i}{\bigcup}\Pi_i)=\underset{i}{\bigcup}\mu_i$ and $\mu_s \otimes\mu_u(J(S_g))=\mu_s\otimes\mu_u(\underset{i}{\bigcup}(J(\Pi_i)))=\nu_s\nu_u(\underset{i}{\bigcup} \mu_i)$. Since $J(S_g)=S_g$, it follows that $\nu_s\nu_u=1$. Let  $\nu=\nu_s$.

     Consider the case $\nu\neq1$. The homeomorphism $J$ has a pair of invariant transversal foliations $\mathcal F^s_P$, $\mathcal F^u_P$ with a common set of saddle singularities having at least three separatrices, and transversal measures $ \mu_s$, $ \mu_u$ such that that $ \mu_s(J(\alpha))=\nu \mu_s(\alpha)$ ($ \mu_u(J(\alpha))=\nu^{-1} \mu_u(\alpha)$) for any arc $\alpha$ transversal to $\mathcal F^s_P$ ($\mathcal F^u_P$). Consequently, for $\nu>1$ ($\nu<1$) the homeomorphism $J$ is a pseudo-Anosov map with dilatation $\nu>1$ ($\frac{1}{\nu}>1$).
   
  Consider the case $\nu=1$. Since the foliation $\mathcal F^s_P$ is invariant under $J$, it follows that separatrices of saddle singularities under the action of $J$ are mapped into separatrices of saddle singularities. Since the set of separatrices is finite, there exists $m\in\mathbb N$ such that $J^m(s_i)=s_i$ and $J^m(l)=l$ for some separatrix $l$ of the saddle singularity $s_i$ of the foliation $\mathcal F^s_P$.
 
Let us prove that $J^m(x)=x$ for any point $x\in l$. Let $[s_i,x]$ be the arc of the curve $l$ bounded by  points $s_i$ and $x$. Since $\mu_u(J^m[s_i,x])=\mu_u([s_i,x])$, it follows that $J^m([s_i,x])=[s_i,x]$. Therefore, $J^m(x)=x$.
 
  Since the leaf $l$ is dense everywhere on $S_g$ (see Statement \ref{denseleaf}) and $J^m|_l=id$, it follows that $J^m(z)=z$ for any $z\in S_g$.

  Consequently, the map  $J$ is  a periodic homeomorphism for $\nu=1$  and   is pseudo-Anosov for $\nu\neq 1$.
\end{proof}   
  
\section{On the model maps}\label{conjofpa}
  
 In this section we prove Theorem \ref{modelmaps} and  auxiliary lemmas.

Recall that a map  $f_2\colon Y\to Y$ of a topological space $Y$ is called a {\it factor} of a map  $f_1\colon X\to X$ of a topological space $X$ if there is a surjective continuous map  $h\colon X \to Y$ such that $hf_1=f_2h$. The map  $h$ is called {\it  semiconjugacy}.

\begin{lem}\label{NWsemiconjugacy}
Let $f_1\colon X\to X$, $f_2\colon Y\to Y$ be homeomorphisms of topological spaces $X$ and $Y$ such that $f_2$ is a factor of $f_1$ with  semiconjugacy $h\colon X \to Y$. Then: \begin{enumerate}
  
   \item $h(NW(f_1))\subset NW(f_2)$;

   \item if $f_2^k(V_y)=V_y$ for some $k\in\mathbb N$, $V_y\subset Y$, then $f_1^k(V_x)\subset V_x$;
  
   \item if $f_1^k(V_x)=V_x$ for some $k\in\mathbb N$, $V_x\subset X$, then $f_2^k(V_y)=V_y$, where $V_y=h(V_x )$.
     \end{enumerate}
\end{lem}

\begin{proof} Let $f_1\colon X\to X$, $f_2\colon Y\to Y$ be homeomorphisms of topological spaces $X$ and $Y$ such that $f_2$ is a factor of $f_1$ with semiconjugacy $h\colon X\to Y$, that is, $hf_1=f_2h$. Let us prove each point of the lemma separately.

\begin{enumerate}
\item Consider the point $x\in NW(f_1)$ and the point $y=h(x)$ with an arbitrary open neighborhood $U_y$. Let  $U_x=h^{-1}(U_y)$. Since $h$ is a continuous map, the inverse image $U_x$ of the open set $U_y$ is also open. Then, by the definition of a non-wandering point $x$, there exists $n\in\mathbb N$ such that $f_1^n(U_x)\cap U_x\neq\emptyset$. Let  $f_1^n(U_x)\cap U_x=\hat U_x$ and $\hat U_y=h(\hat U_x)$. Since $\hat U_x\subset U_x$, then $h(\hat U_x)\subset h(U_x)$, that is, $\hat U_y\subset U_y$. Note that $hf_1^n=f_2^nh$. Since $\hat U_y\subset h(f_1^n(U_x))$, then $\hat U_y\subset f_2^n(h(U_x))=f_2^n(U_y)$. Therefore, $f_2^n(U_y)\cap U_y\neq\emptyset$. Thus, $y=h(x)\in NW(f_2)$.

\item Let $f_2^k(V_y)=V_y$, where $k\in\mathbb N$, $V_y\subset Y$, $V_x=h^{-1}(V_y)$ and $f_1^k(V_x)=V_x'$. Then $f_2^k(h(V_x))=f_2^k(V_y)=V_y$ and $h(f_1^k(V_x))=h(V_x')$. Since $hf_1^k=f_2^kh$, it follows that $h(V_x')=V_y$. Therefore, $V_x'\subset V_x$, that is, $f_1^k(V_x)\subset V_x$.
 
\item Let $f_1^k(V_x)=V_x$, where $k\in\mathbb N$, $V_x\subset X$ and $V_y=h(V_x)$. Then $h(f_1^k(V_x))=h(V_x)=V_y$. Since $hf_1^k=f_2^kh$, then $f_2^k(h(V_x))=f_2^k(V_y)=V_y$. Therefore, $f^k(V_y)= V_y$.
 \end{enumerate} \end{proof}

We will  call a set of numbers $n,k,l$ {\it correct} if $n,k\in\mathbb N$, $l\in\mathbb Z$, where $l=0$ for $k=1$  and $l\in\{1,\dots,k-1\}$  is coprime to $k$ for $k>1$. Everywhere else in this section the set of numbers $n,k,l$ is correct. Let us recall main notation and formulas.

\begin{itemize}

\item
The manifold $M_J$ is the quotient space of   $S_g\times\mathbb R$ under the action of the group $\Gamma=\{\gamma^i, i\in\mathbb Z\}$ of degrees of homeomorphism $\gamma\colon S_g\times\mathbb R \to S_g\times\mathbb R$  given by the formula  $\gamma(z,r)=(J(z),r-1)$, where $J\colon S_g\to S_g$ is an orientation-preserving homeomorphism;

\item $p_{J} \colon S_g\times \mathbb R \to M_{{J}}$ is the natural projection inducing the homomorhisms $\eta_{M_J}\colon M_J\to\mathbb Z$;

\item $\bar\varphi_{n,k,l}\colon \mathbb R\to\mathbb R$ is the diffeomorphism given by the formula \begin{equation}\label{difline} \bar\varphi_{n,k, l}(r)=r+\frac{1}{4\pi nk} sin(2\pi nkr)+\frac{l}{k}; \end{equation}

\item $\mathbb S^1=\{e^{i2\pi\theta}|0\leq \theta <1\}$, $p\colon \mathbb R\to\mathbb S^1$ is the covering, given by the formula $p(r)=s$, where $s=e^{i2\pi r}$;

\item $\varphi_{n,k,l}\colon \mathbb S^1\to\mathbb S^1$is the diffeomorphism given by the formula \begin{equation}\label{difcircle} \varphi_{n,k, l}(s)=p(\bar\varphi_{n,k,l}(p^{-1}(s))); \end{equation}

  \item $\bar\varphi=\bar\varphi_{P,J,n,k,l}(z,r)\colon S_g\times\mathbb R\to S_g\times\mathbb R$ is the homeomorphism given by the formula \begin{equation}\label{bar} \bar\varphi(z,r)=(P(z),\bar\varphi_{n,k,l}(r)), \end{equation} where $ P\colon S_g\to S_g$ is an orientation-preserving pseudo-Anosov homeomorphism such that $J\in Z(P)$;

\item model homeomorphism $\varphi=\varphi_{P,J,n,k,l}\colon M_J\to M_J$ is given by the formula \begin{equation}\label{map} \varphi(w)=p_{_J} (\bar \varphi(p_{_J}^{-1}(w)));\end{equation}
  \item $\Phi$ is a set of model homeomorphisms.
\end{itemize}

Let us introduce the following notation:

\begin{itemize}

\item $\mathcal B_i=p_{_{J}}(S_g\times \{\frac{i}{2nk}\}) \in M_J$ ($i\in\{0,\dots,2nk-1 \}$);

\item $b_i=p(\frac{i}{2nk})\in\mathbb S^1$ ($i\in\{0,\dots,2nk-1\}$);

\item $ p_{J,r}\colon S_g\times\{r\}\to p_{_{J}}(S_g\times\{r\}) $ is the homeomorphism given by the formula \begin{equation} \label{pJr} p_{J,r}=p_J|_{S_g\times\{r\}},\; r \in\mathbb R; \end{equation}

\item $\rho\colon S_g\times\mathbb R\to S_g$ is the canonical projection given by the formula \begin{equation}\label{rho} \rho(z,r)=z; \end{equation}

\item $\rho_r\colon S_g\times\{r\}\to S_g$ is the homeomorphism given by the formula \begin{equation}\label{rhor} \rho_r=\rho|_{S_g\times\{r\}},\; r\in\mathbb R .\end{equation}

\end{itemize}

Note that the Eq. \eqref{map} is obtained from the relation \begin{equation}\label{varphi} p_{_{J}}\bar \varphi=\varphi p_{_{J}},\end{equation} and Eq. \eqref{difcircle} is obtained  from the relation \begin{equation}\label{varphinkl} p\bar \varphi_{n,k,l}=\varphi_{n,k,l} p.\end{equation} Since $ p_{_{J}} \colon S_g\times \mathbb R \to M_{{J}}$ is a natural projection, it follows that \begin{equation}\label{gamma} p_{_{J}}\gamma= p_{_{J}}.\end{equation}

Denote by $h_J\colon M_J\to \mathbb S^1$   the continuous surjective map  given by the formula \begin{equation}\label{hJ} h_J(w)=p(r), \text{\;where\;} w=p_J(z,r)\in M_J. \end{equation} It is readily verified that $h_J\varphi=\varphi_{n,k,l}p_J$. Thus, the following lemma is true.

\begin{lem}\label{semiconjugacy}
  The homeomorphism $\varphi_{n,k,l}\colon\mathbb S^1\to\mathbb S^1$ is the factor of the homeomorphism $\varphi\colon M_J\to M_J$ with semiconjugacy $h_J\colon M_J\to\mathbb S^1$.
\end{lem}

It is directly verified (see Eqs. \eqref{difline} and \eqref{difcircle}) that the non-wandering set of the diffeomorphism $\varphi_{n,k,l}$ consists of $2nk$ points $b_0,\dots,b_{2nk- 1}$ of period $k$ such that points with odd indices $i$ are sinks  and points with even indices are source.

Let us prove Theorem \ref{modelmaps}, that is, we prove the inclusion  $\Phi\subset\mathcal G$.
 
  \begin{proof} Consider the model homeomorphism $\varphi=\varphi_{P,J,n,k,l}\colon M_J\to M_J$. Since the homeomorphism $J$ preserves orientation, it follows that the manifold $M_J$ is orientable. Preserving orientation of  homeomorphisms $P$ and $\varphi_{n,k,l}$ implies preserving orientation by  homeomorphism $\varphi$ inducing by map $\bar\varphi(z,r)=(P(z),\bar \varphi_{n,k ,l}(r))$.

  Let us prove that the connected component $\mathcal B_i$ ($i\in\{0,\dots,2nk-1\}$) is a cylindrical embedding of the surface $S_g$. For $i\in\{0,\dots,2nk-1\}$ we set $\bar U_i=S_g\times[\frac{i}{2nk}-\frac{i}{4nk},\frac{i }{2nk}+\frac{i}{4nk}]$ and $U_i=p_J(\bar U_i)$. Since $p_{_{J}} \colon S_g\times \mathbb R \to M_{{J}}$ is a covering, it follows that  for any $i\in\{0,\dots,2nk-1\} $ its restriction $p_{_{J}}|_{\bar U_i}\colon {\bar U_i}\to U_i$ is a homeomorphism. In addition,  $p_{_{J}}|_{\bar U_i}(S_g\times\{\frac{i}{2nk}\})=\mathcal B_i$. Therefore, $\mathcal B_i$ ($i\in\{0,\dots,2nk-1\}$) is a cylindrical embedding of $S_g$.

Let us prove that $\varphi^{k}(\mathcal B_i)=\mathcal B_i$, $\varphi^{\tilde k_i}( \mathcal B_i)\neq \mathcal B_i$ ($i\in\{0, \dots,2nk-1\}$) for any natural number $\tilde k_i< k$. In accordance with Lemma \ref{semiconjugacy}, the map $\varphi_{n,k,l}$ is the factor of a homeomorphism  $\varphi$ with semiconjugacy $h_J$. Note that $h_J^{-1}(b_i)=\mathcal B_i$ $(i\in\{0,\dots,2nk-1\})$, where $b_i\in\mathbb S^1$ is a point of period $k$. It follows from Lemma \ref{NWsemiconjugacy} that $\varphi^k(\mathcal B_i)\subset \mathcal B_i$. Since the map  $\varphi^k$ is a homeomorphism and the component $\mathcal B_i$ is homeomorphic to $S_g$, it follows that $\varphi^k(\mathcal B_i)=\mathcal B_i$. Suppose that $\varphi^{\tilde k}(\mathcal B_i)=\mathcal B_i$ for some natural number $\tilde k<k$. Then   Lemma \ref{NWsemiconjugacy} implies that $\varphi_{n,k,l}^{\tilde k}(b_i)=b_i$. We come to contradiction that point $b_i$ has period $k$.
  
  Let us prove that the map $\varphi^{k}|_{\mathcal B_i}$ $(i\in\{0,\dots,2nk-1\})$ is topologically conjugate to the orientation-preserving pseudo-Anosov homeomorphism. Since \begin{equation}\label{barJlPk} \gamma^l\Big (\bar \varphi^k\Big (z,\frac{i}{2nk}\Big)\Big)=\Big(J^ l\Big(P^k(z)\Big),\frac{i}{2nk}\Big),\end{equation} it follows that \begin{equation}\label{JlPk} \rho_{\frac{i} {2nk}}\Big(\gamma^l\Big(\bar\varphi^k\Big(\rho_{\frac{i}{2nk}}^{-1}(z)\Big)\Big)\Big)=J^l\Big(P^k(z)\Big). \end{equation} For any point $w\in\mathcal B_i$ we get $\varphi^k(w)\overset{\eqref{map}}{=}p_{_J}(\bar \varphi^k(p_ {_J}^{-1}(w)))\overset{\eqref{gamma}}{=}p_{_J}(\gamma^l(\bar \varphi^k(p_{_J}^{-1 }(w))))\overset{\eqref{barJlPk}}{=}\\ p_{J,\frac{i}{2nk}}(\gamma^l(\bar \varphi^k(p_{J,\frac{i}{2nk}}^{-1}(w))))\overset{\eqref{JlPk}}{=}
  p_{J,\frac{i}{2nk}}(\rho_{\frac {i}{2nk}}^{-1}(J^l(P^k(\rho_{\frac{i}{2nk}}(p_{J,\frac{i}{2nk}}^{- 1}(w))))))$. Consequently, the homeomorphism $\varphi^k|_{\mathcal B_i}$ is topologically conjugate to the orientation-preserving pseudo-Anosov homeomorphism $J^lP^k$ via the homeomorphism $p_{J,\frac{i}{2nk}}\rho_{\frac {i}{2nk}}^{-1}$.

      Lemmas \ref{NWsemiconjugacy} and  \ref{semiconjugacy} imply that $NW(\varphi)\subset(\mathcal B_0\cup\dots\cup\mathcal B_{2nk-1})$.

Since the set of periodic points of a pseudo-Anosov homeomorphism is dense everywhere on the surface (Proposition \ref{dense}) and $\varphi^{k}(\mathcal B_i)=\mathcal B_i$ $(i\in\{0,\dots,2nk -1\})$, it follows that $NW(\varphi)=\mathcal B_0\cup\dots\cup\mathcal B_{2nk-1}$.

Let us prove that the connected components $\mathcal B_i$ with odd indices $i$ belong to the set of attractors of the homeomorphism $\varphi$. Points $b_i$ with odd indices $i$ are sink points of the diffeomorphism $\varphi_{n,k,l}^k$. Therefore, 
  $\varphi^k(u_i)\subset int\; u_i$ and $\underset{j\geq 0}{\bigcap}\varphi_{n,k,l}^{jk}(u_i)=b_i$ for the neighborhood $u_i=h_J(U_i)=p([\frac{i}{2nk}-\frac{i}{4nk},\frac{i}{2nk}+\frac{i}{4nk} ])$ of point $b_i$ with odd index $i$. Since $h_J^{-1}(p[a,b])=p_J(S_g\times[a,b])$ for any $a,b\in\mathbb R$, $h_J\varphi^{jk }=\varphi_{n,k,l}^{jk}h_J$ and $h_J^{-1}(b_i)=\mathcal B_i$, it follows that $\varphi^k(U_i)\subset int\; U_i$, $\underset{j\geq 0}{\bigcap}\varphi^{jk}(U_i)=\mathcal B_i$. Consequently, connected components $\mathcal B_i$ with odd indices $i$ are attractors of the map $\varphi^k$.
 
Analogously one proves that connected components $\mathcal B_i$ with even indices $i$ belong to the set of repellers.

Thus $\varphi\in\mathcal G$. \end{proof}

   \section{The ambient $\Omega$-conjugacy of a homeomorphism $f\in\mathcal G$ to a model map}
 
Recall that the set $\Phi$ consists of model homeomorphisms of the form $\varphi_{P,J,n,k,l}$.
   This section contains a proof of $\Omega$-conjugacy of homeomorphisms of the class $\mathcal G$ with homeomorphisms of the set $\Phi$ and auxiliary lemmas. We will also use the notation introduced in the  Section  \ref{proofprJ} below.

   Let us denote by $\mathcal H$  the set of all homeomorphisms $f$ satisfying the following conditions:
   \begin{enumerate}
  
   \item there exists an orientation-preserving homeomorphism $J\colon S_g\to S_g$ such that $f\colon M_J\to M_J$;
  
\item $f$ preserves the orientation of $M_J$;

\item there exists $m\in\mathbb N$ such that the non-wandering set $NW(f)$ of the homeomorphism $f$ consists of $2m$ connected components $\mathcal B_{0}\cup\dots\cup\mathcal B_{ 2m-1}$;

\item for any $i\in\{0,\dots,2m-1\}$ there is a natural number $k_i$ such that $f^{k_i}(\mathcal B_i)= \mathcal B_i$, $f^ {\tilde k_i}(\mathcal B_i)\neq \mathcal B_i$ for any natural $\tilde k_i< k_i$ and the map  $f^{k_i}|_{\mathcal B_i}$ preserves the orientation of $\mathcal B_i$ ;

\item $f(\mathcal B_i)=\mathcal B_j$, where the numbers $i,j\in\{0,\dots,2m-1\}$ are either even or odd at the same time.

     \end{enumerate}

Note that homeomorphisms of the set $\Phi$ belong to the class $\mathcal H$.
 
For $m\in\mathbb N$ we denote by ${\mathcal T}_{m}$ the set ${\mathcal T}_{m}=\{\frac{i}{2m}, i\in\mathbb Z\}$. Then $p_J^{-1}(NW(f))=S_g\times{\mathcal T}_{m}$, where $f\in\mathcal H$.

  \begin{lem}\label{liftH} For any homeomorphism $f\in\mathcal H$ with non-wandering set consisting of $2m$ connected components, there exist and unique correct set of numbers $n,k,l$ and a lift $\bar f\colon S_g\times\mathbb R\to S_g\times\mathbb R$ such that $$\bar f(z,r)=\Big(f_{r}(z),r+\frac{l} {k}\Big),\;\forall r\in {\mathcal T}_{nk},$$ where $nk=m$ and $f_r\colon S_g\to S_g$ is an orientation-preserving homeomorphism given by $$f_r=\rho_{r+\frac{l}{k}}\bar f\rho_{r}^{-1}.$$

  \end{lem}
 
\begin{proof} Let $f\colon M_J\to M_J$ be a homeomorphism from the class $\mathcal H$.

   Let us prove that there is a lift $\bar f\colon S_g\times\mathbb R\to S_g\times\mathbb R$ of the homeomorphism $f$. By Statement \ref{cyclic} it sufficies to show that $\eta_{M_J}=\eta_{M_J}f_*$.

   Consider the loop $c\in M_J$ which is the projection of the curve $\bar c\in S_g\times\mathbb R$ ($p_J(\bar c)=c$), bounded by  points $\bar c(0)=( z,1)$, $\bar c(1)=\gamma(\bar c(0))=(J(z),0)$ and intersecting each set $S_g\times\{\frac{i}{ 2m}\}$, $i\in\{0,\dots,2m-1\}$ at exactly one point. By construction, the curve $c$ intersects each connected component $\mathcal B_{0},\dots,\mathcal B_{2m-1}$ at exactly one point and $\eta_{M_J}([c])=1$. We set $C=f(c)$ and $C(0)=f(c(0))$. Since $f$ is a homeomorphism such that $f(\mathcal B_i)=\mathcal B_{i'}$, $i,i'\in\{0,\dots,2m-1\}$, it follows that the curve $C=f(c)$ also intersects each component of $\mathcal B_{0},\dots,\mathcal B_{2m-1}$ at exactly one point. We set $\mathcal B_j=f(\mathcal B_{0})$. Choosing a point $\bar C(0)\in p_J^{-1}(C(0))$ such that $\bar C(0)\in S_g\times\{\frac{j}{2m}+ 1\}$ by the monodromy theorem there is a unique lift $\bar C$ of the path $C$ starting at the point $\bar C(0)$. Since the loop $C$ intersects each component $\mathcal B_{0},\dots,\mathcal B_{2m-1}$ at exactly one point, it follows that there are 2 cases: 1) $\bar C(1)= \gamma^{-1}(\bar C(0))$, 2) $\bar C(1)=\gamma(\bar C(0))$.

    Let us show that the case 1) is not realized.
  
    Consider the case $m=1$. Then $f(\mathcal B_0)=\mathcal B_0$. Since the homeomorphism $f$ preserves the orientation $M_J$ and the orientation $\mathcal B_0$, it follows that the curve $C(t)$ must be parameterized in one direction with  the parameterization  of the curve $c(t)$ with respect to the surface $\mathcal B_0$. Thus $\bar C(1)=\gamma(\bar C(0))$.

   Consider the case $m>1$. Let us denote by $\xi_c\colon \mathbb S^1\to c$, $\xi_C\colon \mathbb S^1\to C$ homeomorphisms such that $\xi_c(b_i)=\mathcal B_i\cap c$, $\xi_C(b_i)=\mathcal B_i\cap C$, where $i\in\{0,\dots,2m-1\}$. Define the homeomorphism $\psi\colon\mathbb S^1\to\mathbb S^1$ by the formula $\psi=\xi_C^{-1}f\xi_c$. Let us prove that the homeomorphism $\psi$ preserves orientation. Assume the converse. Let us prove that there exists $q\in\{0,\dots,2m-1\}$ such that $\psi(b_q)=b_q$. Let $\mathcal B_j=f(\mathcal B_0)$. Then $\psi(b_0)=b_j$. If $j=0$, then $q=0$. Let $j\neq 0$. By the condition of the class $\mathcal H$, the number $j$ is even. Since $\psi$ by assumption changes the orientation of $\mathbb S^1$ and the set $b_0\cup\dots\cup b_{2m-1}$ is invariant, it follows that the arc of the circle $(b_0,b_j)$ is mapped into itself and $\psi(b_i)=b_{j-i}$, $i\in\{0,\dots,\frac{j}{2}\} $. Thus $\psi(b_\frac{j}{2})=b_{\frac{j}{2}}$ and $q=\frac{j}{2}$. Therefore, $f(\mathcal B_q)=\mathcal B_q$. Since $\psi$ changes orientation, it follows that   the curve $C(t)$ is parameterized in the direction opposite to the parameterization of the curve $c(t)$ with respect to the surface $\mathcal B_q$ (see Fig. \ref{orientation}). Since the homeomorphism $f$ preserves the orientation $M_J$ and the orientation $\mathcal B_q$, then    the parameterization of the curve  $C(t)$ must be parameterized in one direction with the parameterization of the curve $c(t)$ with respect to the surface $\mathcal B_q$. We got a contradiction. Consequently, the homeomorphism $\psi$ preserves the orientation of $\mathbb S^1$. Then $\bar C(1)=\gamma(\bar C(0))$.

   \begin{figure}[!ht]
\includegraphics[scale=2]{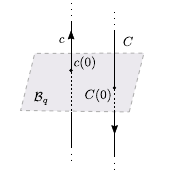}
\caption{Direction of increasing parameter $t\in[0,1]$ on   curves $c$ and $C$.}
\label{orientation}
\end{figure}

   Thus $\bar C(1)=\gamma(\bar C(0))$ and $\eta_{M_J}(f_*([c]))=1$. Consequently, $\eta_{M_J}=\eta_{M_J}f_*$ and there is a unique lift $\bar f\colon S_g\times\mathbb R\to S_g\times\mathbb R$ of the homeomorphism $f$ such that $ \bar f(\bar c(1))=\bar C(1)$ and \begin{equation}\label{fjjf}\bar f\gamma=\gamma\bar f.\end{equation}
  
   Let us  find the correct set of numbers $n,k,l$ for the homeomorphism $f$.
The case $m=1$ corresponds to the correct set of numbers $n=1$, $k=1$ and $l=0$. Consider  the case $m>1$.
  Since the homeomorphism $\psi$ is orientation preserving, it follows that  it has a rational rotation number $\frac{l}{k}$, where $k\in\mathbb N$, $l\in\{0,\dots,k-1\}$ and $(l,k)=1$ (see \cite[Theorem 4.1]{Sanhueza}). From \cite[Theorem 4.2]{Sanhueza} it follows that all periodic points of the homeomorphism $\psi$ have period $k$. Since point $b_i$ with even (odd) index $i$ is mapped to point $b_{i'}$ with even (odd) index $i'$, it follows that $2m$ points $b_0,\dots,b_{2m -1}$ are divided into 2  invariant sets of equal power, each of which consists of points of period $k$. Therefore, $m$  is divisible by $k$. We set $n=\frac{m}{k}$. Thus $n,k,l$ is the required correct set of numbers.
 
  Since the rotation number of $\psi$ is equal to $\frac{l}{k}$, it follows that  $\psi(b_0)=b_{2nl}$, that is, $f(\mathcal B_0)=\mathcal B_{2nl}$.

Let us find a formula that defines the map $\bar f$ for the point $(z,r)\in S_g\times\mathcal T_{nk}$. Since $\bar C(1)=\gamma(\bar C(0))$, it follows that $\bar C(1)\in S_g\times\{\frac{2nl}{2nk}\}=S_g\times\{\frac{l}{k}\}$. Invariance of the set $p_J^{-1}(NW(f))=S_g\times{\mathcal T}_{nk}$  under $\bar f$ implies that $\bar f(S_g\times[0, 1])=S_g\times[\frac{l}{k},1+\frac{l}{k}]$, where $\bar f(S_g\times\{0\})=S_g\times\{\frac{l}{k}\}$. From here we get that $\bar f(S_g\times\{\frac{i}{2nk}\})=S_g\times\{\frac{i}{2nk}+\frac{l}{k}\} $ for any $i\in\{0,\dots,2nk-1\}$. Using Eq. \eqref{fjjf} we obtain that $\bar f=\gamma^m\bar f \gamma^{-m}$ for any $m\in\mathbb Z$. Then  $\bar f(S_g\times\{r\})=\gamma^{[r]}(\bar f( \gamma^{-[r]}(S_g\times\{r\})))$, where $[r]$ is the integer part of the number $r\in\mathbb R$. Thus it is readily verified that $\bar f(S_g\times\{r\})=S_g\times\{r+\frac{l}{k}\}$ for $r\in {\mathcal T}_ {nk}$. Then for any $r\in {\mathcal T}_{nk}$ the homeomorphism $f_r\colon S_g\to S_g$ is correctly defined and given by the formula $f_r=\rho_{r+\frac{l}{k}}\bar f\rho_{r}^{-1}$. Thus $\bar f(z,r)=(f_r(z),r+\frac{l}{k})$ for any $r\in\mathcal T_{nk}$.

   It remains to prove that $f_r$ preserves the orientation of $S_g$, where $r\in\mathcal T_{nk}$. Preserving orientation of $M_J$ by $f$ implies preserving orientation of $S_g\times\mathbb R$ by its lift $\bar f$. Since $\bar f(S_g\times\{r\}=f_r(S_g)\times\{r+\frac{l}{k}\}$ for any $r\in\mathcal T_{nk}$, it follows that the homeomorphism $\bar f$ preserves the orientation of $\mathbb R$. Therefore, $\bar f$ preserves the orientation of $S_g$, that is, $f_r$ preserves the orientation of $S_g$. \end{proof}


Note that in the case $f=\varphi_{P,J,n,k,l}$ the equality $f_r(z)=P(z)$ holds for any $r\in{\mathcal T}_{nk} $ and $\bar f=\bar \varphi_{P,J,n,k,l}$.

        \begin{lem}\label{gomf0}

  Let $f\in\mathcal H$. Then $f_r$ is isotopic to $f_0$ for any $r\in{\mathcal T}_{nk}$.

   \end{lem}
  
    \begin{proof} Let $f\in\mathcal H$. Let us prove that $f_r$ is isotopic to $f_0$ for any $r\in{\mathcal T}_{nk}$.

    Define a family of continuous maps $F_{r,t}\colon S_g\to S_g$ by the formula $F_{r,t}(z)=\rho(\bar f(z,r t))$, where $t\in[0,1]$, $r\in{\mathcal T}_{nk}$. Then $F_{r,t}$ defines a homotopy connecting the maps $F_{r,0}=f_0$ and $F_{r,1}=f_{r}$. Thus,  homeomorphisms $f_0$ and $f_{r}$ are homotopic. It follows from \cite[p. 5.15]{Chishang}  that they are isotopic  for any $r\in{\mathcal T}_{nk}$. \end{proof}

  \begin{lem}\label{vsp}

  Let $f\colon M^3\to M^3$ be a homeomorphism from the class $\mathcal G$. Then there exists a homeomorphism $f'\in\mathcal H$ is topologically conjugate to $f$.

   \end{lem}

   \begin{proof}
  Let $f\colon M^3\to M^3$ be a homeomorphism from the class $\mathcal G$ with non-wandering set consisting of $q$ connected components $B_0,\dots, B_{q-1}$.

   In accordance with \cite[Lemma 2.1]{GrPoCh}, the set $M^3\setminus(B_0\cup\dots\cup B_{q-1})$ consists of $q$ connected components $V_0,\dots,V_{q -1}$, bounded by  one connected component of an attractor and one connected component of a repeller. Therefore, $q=2m$, where $m\in\mathbb N$. Without loss of generality, for $m>1$ we can assume that $cl\;V_i\cap cl\;V_{i-1}=B_{i-1}$, where $i\in\{1,\dots ,2m-2\}$ and $cl\;V_0\cap cl\;V_{2m-1}=B_{2m-1}$.

     In accordance with  \cite[Lemma 2.2]{GrPoCh}, each connected component $V_i$, $i\in\{0,\dots,2m-1\}$ of the set $M^3\setminus(B_0\cup\dots\cup B_{2m-1})$ is homeomorphic to $S_g\times[0,1]$. It follows from \cite[Lemma 2]{GrinesLevchenkoPochinka2014} that there exists a continuous surjective map $H\colon S_g\times[0,1]\to M^3$ (see Fig. \ref{components}) such that  maps $H|_{S_g\times\{\frac{i}{m}\}}\colon S_g\times\{\frac{i}{m}\}\to B_i$ ($i\in\{0 ,\dots,2m-1\}$), $H|_{S_g\times\{1\}}\colon S_g\times\{1\} \to B_{0}$ and $H|_{S_g \times(0,1)}\colon S_g\times(0,1) \to M^3\setminus B_{0}$ are homeomorphisms.
\begin{figure}[!ht]
\includegraphics[scale=1.5]{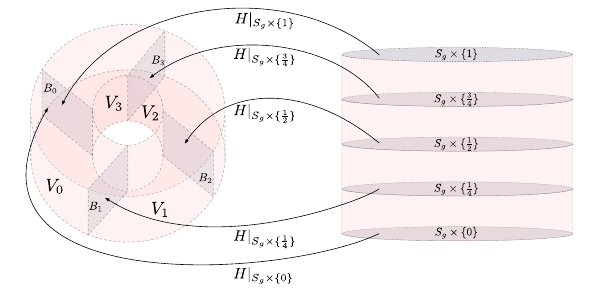}
\caption{Action of the homeomorphism $H$ in the case $m=2$.}
\label{components}
\end{figure}

   Let $J(z)=\rho_0((H|_{S_g\times\{0\}})^{-1}(H|_{S_g\times\{1\}}(\rho_1^{- 1}(z))))$ (see Fig. \ref{componentsJ}).
  
    \begin{figure}[!ht]
\includegraphics[scale=1.5]{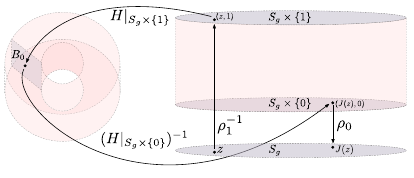}
\caption{Homeomorphism $J\colon S_g\to S_g$.}
\label{componentsJ}
\end{figure}

     Denote by $[r]$   the integer part of the number $r\in\mathbb R$. Define a continuous map $h\colon S_g\times\mathbb R\to M^3$ by the formula $h(z,r)=H(\gamma^{[r]}(z,r))$.
    
     Let  the homeomorphism $\xi\colon M^3\to M_J$ be given by the formula $\xi=p_{_{J}}(h^{-1}(w))$. Set $f'=\xi f \xi^{-1}$.

  Let us prove that the homeomorphism $f'$ satisfies all 5 conditions of the class $\mathcal H$. Since $M^3$ is orientable and homeomorphic to $M_J$, it follows that $J$ preserves the orientation of $S_g$ and condition 1 is satisfied. Since $f$ preserves the orientation of $M^3$, it follows that $f'$ preserves the orientation of $M_J $ and condition 2 is satisfied. Since $\xi(NW(f))=NW(f')$ and $h^{-1}(NW(f))=S_g\times{\mathcal T}_{nk }$, it follows that $NW(f')=p_{J}(h^{-1}(NW(f)))=p_J(S_g\times{\mathcal T}_{nk})=\mathcal B_0\cup\dots\cup\mathcal B_{2m-1}$. Therefore, condition 3 is satisfied. Since for any $B_i$ ($i\in\{0,\dots,2m-1\}$) there is a natural number $k_i$ such that $f^{k_i}( B_i) = B_i$, $f^{\tilde k_i}( B_i)\neq B_i$ for any natural $\tilde k_i< k_i$ and the map $f^{k_i}|_{ B_i}$ preserves the orientation of $B_i$, it follows that the same is true for the connected component $\mathcal B_i$ of the non-wandering set  $NW(f')$, that is, condition 4 is satisfied. The connected components of the non-wandering set $NW(f)$ of the homeomorphism $f$ are numbered in such a way that if $B_i$  is the connected component of an attractor of the homeomorphism $f$, then $B_{i+1 \pmod {2m}}$ is the connected component of a repeller of the homeomorphism $f$. Therefore, $f(B_i)=B_j$, where $i,j\in\{0,\dots,2m-1\}$ are either even or odd at the same time. Since $\xi(B_i)=\mathcal B_i$ ($i\in\{0,\dots,2m-1\}$), it follows that $f'(\mathcal B_i)=\mathcal B_j$, where $ i,j\in\{0,\dots,2m-1\}$ are simultaneously either even or odd, that is, condition 5 is satisfied. Thus, $f'\in\mathcal H$. \end{proof}

Everywhere below in this section we mean by $\bar f$, $f_r$ and $n,k,l$  the lift of the homeomorphism $f\in\mathcal H$, the homeomorphism $f_r\colon S_g\to S_g$, $r\in\mathcal T_ {nk}$, and the correct set of numbers $n,k,l$ from Lemma \ref{liftH}.

       \begin{lem}\label{gomf02}

  Let $f\in\mathcal H\cap\mathcal G$. Then $f_0$ is isotopic either to some periodic homeomorphism or to some pseudo-Anosov homeomorphism.

   \end{lem}
  
    \begin{proof} Let $f\in\mathcal H\cap\mathcal G$.

        Let us prove that $f_0$ is isotopic either to some periodic homeomorphism or to some pseudo-Anosov homeomorphism.
       
          Since $\bar f$ is a lift of a homeomorphism $f$, it follows that \begin{equation}\label{lift1} p_J\bar f=fp_J.\end{equation}
      Therefore,\begin{equation}\label{formulaf} f(w)=p_J(\bar f( p_J^{-1}(w))). \end{equation}
       
       For $r\in {\mathcal T}_{nk}$  denote by $\phi_r\colon S_g\to S_g$ the homeomorphism given by the formula \begin{equation}\label{phirSg}\phi_r=J^{l}f_{r+\frac{(k-1)l}{k}} \cdots f_{r+\frac{l}{k}}f_r.\end{equation}
        Then it is readily verified that \begin{equation}\label{phirgamma} \gamma^{l}(\bar f^k|_{S_g\times{\mathcal T}_{nk}}(z,r))= (\phi_r(z),r),\;\text{ where}\; r\in {\mathcal T}_{nk}.\end{equation} Therefore,\begin{equation}\label{phir} \phi_r=\rho_r\gamma^l\bar f^k\rho_r^{-1}. \end{equation} Thus, $f^k|_{\mathcal B_0}(w)\overset{\eqref{formulaf}}{=}p_J(\bar f^k( p_J^{-1}(w )))\overset{\eqref{gamma}}{=}p_J(\gamma^l(\bar f^k( p_J^{-1}(w)))) \overset{\eqref{phirgamma}}{ =} p_{J,0}(\gamma^l(\bar f^k( p_{J,0}^{-1}(w))))\overset{\eqref{phir}}{=}\\p_ {J,0}(\rho_0^{-1}(\phi_0(\rho_0( p_{J,0}^{-1}(w)))))$, that is, \begin{equation}\label{liftpA} f^k|_{\mathcal B_0}=p_{J,0}\rho_0^{-1}\phi_0\rho_0p_{J,0}^{-1}.\end{equation} Therefore, the homeomorphism $ \phi_0$ is topologically conjugate to the homeomorphism $f^k|_{\mathcal B_0}$ via the map $p_{J,0}\rho_0^{-1}$. Since the homeomorphism $f^k|_{\mathcal B_0}$ is topologically conjugate to the pseudo-Anosov homeomorphism, it follows that the homeomorphism $\phi_0$ is also a pseudo-Anosov map (see Statement \ref{pA2}).

   Eq. \eqref{fjjf} implies that $(J(f_r(z)),r+\frac{l}{k}-1)=(f_{r-1}(J(z)),r-1 +\frac{l}{k})$ and \begin{equation}\label{Jf0} Jf_r=f_{r-1}J\;\text{for any}\; r\in{\mathcal T}_{nk}.\end{equation} Therefore, $f_0J^l=J^lf_l$. Then $f_0(J^{l}f_{\frac{(k-1)l}{k}} \cdots f_{\frac{l}{k}}f_0)=(J^{l}f_{l }f_{\frac{(k-1)l}{k}} \cdots f_{\frac{l}{k}})f_0$, that is, \begin{equation}\label{phif0} \phi_0=f_0 ^{-1}\phi_{\frac{l}{k}}f_0.\end{equation} It follows from  Eq. \eqref{phif0} and Statement \ref{pA2}  that $\phi_{\frac{l} {k}}$ is also a pseudo-Anosov homeomorphism.
       
        Since $f_{r}$ is isotopic to $f_0$ for any $r\in{\mathcal T}_{nk}$ by Lemma \ref{gomf0}, it follows that $J^{l}f_{\frac{( k-1)l}{k}} \cdots f_{\frac{l}{k}}f_0$ is isotopic to  $J^{l}f_{l}f_{\frac{(k-1)l}{k }} \cdots f_{\frac{l}{k}}$, that is, $\phi_0$ is isotopic to $\phi_{\frac{l}{k}}$. Then, according to Statement \ref{conjugacypA}, there exists an isotopic to the identity homeomorphism $h\colon S_g\to S_g$ such that \begin{equation}\label{phih} \phi_0=h\phi_{\frac{l}{k }}h^{-1}. \end{equation}  Putting Eq. \eqref{phih} in Eq. \eqref{phif0}, we obtain  that $\phi_0=f_0^{-1}(h^{-1}\phi_0 h)f_0$, that is, $(hf_0 )\phi_0=\phi_0 (h f_0)$.
       
        Since $\phi_0\in\mathcal P$ and $hf_0\in Z(\phi_0)$, it follows that   the homeomorphism $hf_0$ is either periodic or pseudo-Anosov by Theorem \ref{prJ}. Isotopicity to the identity of $h$ implies that $f_0$ is isotopic either to some periodic homeomorphism or to some pseudo-Anosov homeomorphism. \end{proof}

       \begin{lem}\label{conj}

  Let $f\in\mathcal H\cap\mathcal G$ and $f_0$ be isotopic to some periodic homeomorphism. Then there exists a homeomorphism $f'\in\mathcal H$ such that $f'$ is topologically conjugate to $f$ and $f'_0$ is isotopic to some pseudo-Anosov homeomorphism.

   \end{lem}

   \begin{proof}
  Let $f\colon M_J\to M_J$ be a homeomorphism from the class $\mathcal H\cap\mathcal G$ with non-wandering set consisting of $2nk$ connected components of period $k$, and $f_0$ is isotopic  to some periodic homeomorphism.

  Let us show that $k\neq 1$. Assume the converse. Then $l=0$ and the homeomorphism $\phi_0$ has the form $\phi_0=f_0$ (see Eq. \eqref{phirSg}). According to Eq. \eqref{liftpA}, the homeomorphism $\phi_0$ is topologically conjugate to the pseudo-Anosov homeomorphism $f^k|_{\mathcal B_0}$. We come to contradiction with the fact that $k=1$. Therefore, $k>1$.

  Define the homeomorphisms $\bar h, \gamma' \colon S_g\times\mathbb R\to S_g\times\mathbb R$ by the formulas $\bar h(z,r)=(z,-r)$, $\gamma' (z,r)=(J^{-1}(z),r-1)$. Recall that $\gamma(z,r)=(J(z),r-1)$. Since $(J(z),-(r-1))=(J(z),(-r)+1)$, it follows that $\bar h\gamma=(\gamma')^{-1} \bar h$. Therefore, the homeomorphism $\bar h$ projects into the homeomorphism $h\colon M_{J}\to M_{ J^{-1}}$ (see Statement \ref{cyclic}), given by the formula $h=p_{ J^ {-1}}(\bar h(p_{ J}^{-1}(w)))$, where $p_{J^{-1}}\colon S_g\times\mathbb R\to M_{ J ^{-1}}$ is a natural projection.
 
  Set $f'=hfh^{-1}$. Recall that for a homeomorphism $f\in\mathcal H$ there is  a unique lift $\bar f\colon S_g\times\mathbb R\to S_g\times\mathbb R$ such that $\bar f_{S_g\times { \mathcal T}_{nk}}(z,r)=(f_r(z),r+\frac{l}{k})$, where $n,k,l$ is the correct set of numbers. Consider the lift $\bar f'$ of the homeomorphism $f'$ given by the formula $\bar f'=\gamma^{-1}\bar h \bar f \bar h^{-1}$.
  Then for any $r\in{\mathcal T}_{nk}$ we have $\bar f'(z,r)=(J(f_r(z)),r+\frac{k-l}{k})$. Since $k\neq 1$, it follows that $l\in\{1,\dots,k-1\}$. Therefore, $(k-l)\in\{1,\dots,k-1\}$ and coprime to $k$. Thus, $n,k,(k-l)$ is the correct set of numbers and $f'_r=Jf_r$.
 
   Let us prove that the homeomorphism $f'_0$ is isotopic to some pseudo-Anosov homeomorphism. By Lemma \ref{gomf02} the homeomorphism $f'_0$ is isotopic either to some periodic map or to some pseudo-Anosov map. Suppose that the homeomorphism $f'_0=Jf_0$ is isotopic to a periodic homeomorphism. Then the homeomorphism $J=f'_0f_0^{-1}$ is also isotopic to a periodic homeomorphism. Since $J$ and $f_0$ are isotopic to periodic homeomorphisms and, according to Lemma \ref{gomf0}, $f_0$ is isotopic to $f_r$ for any $r\in{\mathcal T}_{nk}$, it follows tha the homeomorphism $\phi_0=J^{l}f_{\frac{(k-1)l}{k}} \cdots f_{\frac{l}{k}}f_0$ is also isotopic to periodic homeomorphism. We come to contradiction with the fact that $\phi_0$ is topologically conjugate to the pseudo-Anosov homeomorphism $f^k|_{\mathcal B_0}$ (see Eq. \eqref{liftpA}). Consequently, the homeomorphism $f'_0$ is isotopic to the pseudo-Anosov homeomorphism. Thus, $f'\in\mathcal H$ is topologically conjugate to $f$ and $f'_0$ is isotopic to some pseudo-Anosov homeomorphism. \end{proof}
  

     \begin{lem}\label{exJ}

Let $f\in \mathcal H\cap\mathcal G$ and $f_0$ be isotopic to some pseudo-Anosov homeomorphism $P$. Then there is a homeomorphism $f'\colon M_{J'}\to M_{J'}$ from the class $ \mathcal H$ such that $f'$ is topologically conjugate to $f$, $J'P=PJ'$ and $f_0'$ is isotopic to $P$.

   \end{lem}

      \begin{proof}
Let $f\colon M_J\to M_J$ be a homeomorphism from the class $ \mathcal H\cap\mathcal G$ and $P$ be a pseudo-Anosov homeomorphism of the surface $S_g$, isotopic to $f_0$.
 
Let us construct a homeomorphism $J'\colon S_g\to S_g$. Set \begin{equation}\label{PJP} P'=J^{-1}PJ. \end{equation}
Denote by $F_{t}$  the isotopy connecting the homeomorphisms $F_{0}=f_0$ and $F_{1}=P$. Then the family of maps $J^{-1}F_tJ$ defines an isotopy connecting the maps $J^{-1}F_{0}J=J^{-1}f_0J=f_1$ and $J^{-1}F_{ 1}J=J^{-1}PJ=P'$. Since $f_0$ is isotopic to $f_1$ (see Lemma \ref{gomf0}) and  to $P$,  $f_1$ is isotopic to $P'$, it follows that $P$ is isotopic to $P'$.
Homeomorphism $P$ is topologically conjugate to the pseudo-Anosov homeomorphism $P'$, $P$ is isotopic to $P'$. Then by Statement \ref{conjugacypA} there exists an isotopic to the identity homeomorphism $\xi$ such that \begin{equation}\label{Peta}P'=\xi P\xi^{-1}. \end{equation} Set \begin{equation}\label{JJ} J'=J\xi, \; \gamma'=(J'(z),r-1).\end{equation}
    Note that $J'P\overset{\eqref{JJ}}{=}J\xi P\overset{\eqref{Peta}}{=}JP'\xi\overset{\eqref{PJP}}{= }PJ\xi\overset{\eqref{JJ}}{=}PJ'$.

  Let us construct a homeomorphism $Y\colon M_J\to M_{J'}$. Denote by $\xi_{t}$   the isotopy connecting the homeomorphism $\xi_{0}= \xi$ and the identity map $\xi_{1}=id$. Define the homeomorphism $y_r\colon S_g\to S_g$ by the formula \begin{equation*} y_r=\begin{cases}
               \xi_{6nk(1-r)} \;& \text{for}\;r\in[1-\frac{1}{6nk},1]; \\
           id \;& \text{for}\; r\in[0.1-\frac{1}{6nk}].
         \end{cases}\end{equation*} Define the homeomorphism $y\colon S_g\times[0,1]\to S_g\times[0,1]$ by the formula $y(z,r)=(y_r(z) ,r)$. Note that \begin{equation}\label{yy}y(z,0)=(z,0)\text{ and}\; y\Big(z,\frac{l}{k}\Big)=\Big(z,\frac{l}{k}\Big). \end{equation} Denote by  $[r]$   the integer part of the number $r\in\mathbb R$. Define the homeomorphism $\bar Y\colon S_g\times\mathbb R\to S_g\times\mathbb R$ by the formula \begin{equation}\label{Ygamma} \bar Y(z,r)=(\gamma')^{-[r]}(y(\gamma^{[r]}(z,r))).\end{equation}
             Since $\gamma' \bar Y=\bar Y\gamma$, it follows that  the homeomorphism $\bar Y$ projects into the homeomorphism $Y\colon M_{J}\to M_{ J'}$ (see Statement \ref{cyclic}), given by the formula $Y=p_{ J'}(\bar Y(p_{ J}^{-1}(w)))$, where $p_{ J}\colon S_g\times\mathbb R\to M_{J}$, $p_{J'}\colon S_g\times\mathbb R\to M_{J'}$ are natural projections.

              Set $f'=YfY^{-1}\colon M_{J'}\to M_{J'}$. By construction $f'\in\mathcal H$. Let us prove that $f'_0$ is isotopic to $P$. Consider the lift \begin{equation}\label{fY} \bar f'=\bar Y \bar f \bar Y^{-1} \end{equation} of the homeomorphism $f$. It is readily verified that $\bar f'(z,r)=(f'_r(z),r+\frac{l}{k})$, where $r\in\mathcal T_{nk}$ and $f '_r$ is a homeomorphism of $S_g$. Let us show that $f'_0=f_0$. Indeed, $\bar f'(z,0)\overset{\eqref{fY}}{=}\bar Y(\bar f(\bar Y^{-1}( z,0)))\overset{\eqref{Ygamma}}{=}\bar Y(\bar f(y^{-1}(z,0)))\overset{\eqref{yy}}{= }\bar Y(\bar f(z,0))=\bar Y(f_0(z),\frac{l}{k}) \overset{\eqref{Ygamma}}{=}
y_{\frac{l}{k}}(f_0(z),\frac{l}{k}) \overset{\eqref{yy}}{=}(f_0(z),\frac{l}{ k}) $. Thus, $f'_0$ is also isotopic to $P$. \end{proof}


      Let us prove that any homeomorphism from the class $\mathcal G$ is ambiently $\Omega$-conjugate to a homeomorphism from the class $\Phi$.

   \begin{proof}
   Let $f\in\mathcal G$.

According to Lemma \ref{vsp}, wihout loss of generality, we may assume that $f$ is defined on $M_{J}=S_g\times\mathbb R/{\Gamma}$ with natural projection $p_{J} \colon S_g\times \mathbb R \to M_{J}$, where $ J$ is a orientation-preserving homeomorphism of the surface $S_g$ and $\Gamma=\{\gamma^i| i\in\mathbb Z\}$ is a group of degrees of the homeomorphism $\gamma\colon S_g\times\mathbb R \to S_g\times\mathbb R$ given by the formula $\gamma(z,r)=(J(z ),r-1)$. It follows from Lemma \ref{liftH} that the non-wandering set of the homeomorphism $f$ consists of $2nk$ connected components $\mathcal B_0,\dots,\mathcal B_{2nk-1}$ and there is a lift $\bar f$ of the homeomorphism $f$ such that $\bar f(z,r )=(f_r(z),r+\frac{l}{k})$ for any $r\in{\mathcal T}_{nk}$, where $f_r\colon S_g\to S_g$ is an orientation preserving  homeomorphism of the surface and $n,k,l$ is the correct set of numbers.

According to Lemmas \ref{gomf0},\ref{gomf02},\ref{conj},\ref{exJ}, without loss of generality we may assume that $f_r$ is isotopic to some orientation-preserving pseudo-Anosov homeomorphism $P$ for any $r\in {\mathcal T}_{nk}$ and $J\in Z(P)$. Since $J$ preserves the orientation of $S_g$, it follows that the homeomorphism $J^lP^k$ also preserves the orientation of $S_g$.

Let us prove that the homeomorphism $J^{l}P^k$ is a pseudo-Anosov homeomorphism.
Using Eqs. \eqref{phirgamma} and \eqref{phir}, we obtain \begin{equation}\label{fk} f^k|_{p_J(S_g\times\{r\})}=p_{J,r }\rho_r^{-1}\phi_r\rho_rp_{J,r}^{-1},\; r\in {\mathcal T}_{nk},\end{equation}
that is, the homeomorphism $\phi_r$ ($r\in{\mathcal T}_{nk}$) is topologically conjugate to the pseudo-Anosov homeomorphism $ f^k|_{p_J(S_g\times\{r\})}$. Since by Lemma \ref{gomf0} the homeomorphism $f_r$ for any $r\in{\mathcal T}_{nk}$ is isotopic to $P$, it follows  that the homeomorphism $\phi_r=J^lf_{r+\frac{( k-1)l}{k}}\cdots f_{r+\frac{l}{k}}f_r$ is isotopic to $J^{l}P^k$, that is, the homeomorphism  $J^{l}P^k $ is isotopic to the pseudo-Anosov homeomorphism. According to Theorem \ref{prJ}, we obtain  
 that the homeomorphism $J^{l}P^k$ is a pseudo-Anosov map.

Note that  homeomorphisms $J^lP^k$ and $\phi_r$ are isotopic for any $r\in\mathcal T_{nk}$ and are pseudo-Anosov homeomorphisms. Then, according to Statement \ref{conjugacypA}, maps $\phi_r$ and $J^{l}P^k$ are topologically conjugate for any $r\in T$ via some isotopic to the identity homeomorphism. Denote such a homeomorphism by $h_r$. Then for any $r\in\mathcal T_{nk}$ we obtain that \begin{equation}\label{hi} J^{l}P^k=h_r(\phi_r)h_r^{-1}.\end{equation}

Thus, each homeomorphism $f\in\mathcal G$ corresponds to the correct set of numbers $n,k,l$ and orientation-preserving homeomorphisms $P\colon S_g\to S_g$, $J\colon S_g\to S_g$ such that the homeomorphisms $P$, $J^lP^k$ are pseudo-Anosov and $J\in Z(P)$. Therefore, there is correctly defined model map  $\varphi_{P,J,n,k,l}\in\Phi$.

Let us prove that the homeomorphism $f$ is ambiently $\Omega$-conjugate to $\varphi_{P,J,n,k,l}$. We construct a homeomorphism $ {f'} \colon M_J \to M_J$, topologically conjugate to $f$ and coinciding with the homeomorphism $\varphi_{P,J,n,k,l}$ on the non-wandering set (${f'}|_{NW({f'})}=\varphi_{P,J,n,k,l}|_{NW(\varphi_{P,J,n,k,l})}$).

We divide the construction   into steps.
   
  {\bf Step 1.} Construct a homeomorphism $x\colon S_g\times U \to S_g\times U $, where $U=\underset{j\in\{0,\dots,k-1\}}{\bigcup}U_j$, $U_j=[-\frac{1}{4nk}-j\frac{l}{k},\frac{1}{k}-\frac{1}{4nk}-j\frac{l}{k})$.

Let $T=\{0,\frac{1}{2nk},\dots,\frac{2n-1}{2nk}\}$. Note that $T=\mathcal T_{nk}\cap U_0$ and $r\in\mathcal T_{nk}\cap U_j$ has the form $r=i-j\frac{l}{k}$, where $j \in\{0,\dots,k-1\}$ and the number $i\in T$ is uniquely determined. For $i\in T$ and $j\in\{0,\dots,k-1\}$ we define the homeomorphism $\xi_{i,j}\colon S_g\to S_g$ by the formula \begin{equation}\label {xi}\xi_{i,j}=
               P^{-j}h_{i}\underbrace{f_{i-j\frac{l}{k}+(j-1)\frac{l}{k}}\cdots f_{i-j\frac{l}{k}}}_{j\;\text{maps}}.
        \end{equation} Since the homeomorphism $f_{i-j\frac{l}{k}+(j-1)\frac{l}{k}}\cdots f_{i-j\frac{l}{k}+\frac{l}{k}}f_{i-j\frac{l}{k}}$ is isotopic to $P^{j}$ for $j\in\{1,\dots,k-1\}$ and the homeomorphism $h_i$ is isotopic to the identity, it follows that the homeomorphism $\xi_{i,j}$ is isotopic to the identity for any $j\in\{0,\dots,k-1\}$. Let $\xi_{i,j,t}$ denote the isotopy connecting the homeomorphism $\xi_{i,j,0}= \xi_{i,j}$ and the identity map $\xi_{i,j,1}= id$.

    For $r\in U$ we define the homeomorphism $x_r\colon S_g\to S_g$ by the formula \begin{equation*}\label{xr} x_r=\begin{cases}
                 \xi_{i,j,6nk|r-(i-j\frac{l}{k})|} \;& \text{for}\; |r-(i-j\frac{l}{k})|\leq\frac{1}{6nk}; \\
           id \;& \text{for others}\; r\in U.
         \end{cases}\end{equation*} Define the homeomorphism $x\colon S_g\times U\to S_g\times U$ by the formula \begin{equation*} x(z,r)=(x_r(z),r) .\end{equation*} Note that \begin{equation}\label{xzr} x\Big(z,i-j\frac{l}{k}\Big)=\Big(\xi_{i,j}( z),i-j\frac{l}{k}\Big).\end{equation}

  {\bf Step 2.} Let us extend the homeomorphism $x\colon S_g\times U\to S_g\times U$ to the homeomorphism $\bar X\colon S_g\times\mathbb R\to S_g\times\mathbb R$.
 
  Let us prove that
for any point $r\in\mathbb R$ there is a unique integer $m\in\mathbb Z$ such that $(r-m)\in U$.

  Divide the half-interval $[-\frac{1}{4nk},1-\frac{1}{4nk})$ into $k$  half-intervals: $[-\frac{1}{4nk},1-\frac{ 1}{4nk})=[-\frac{1}{4nk},\frac{1}{k}-\frac{1}{4nk})\cup[-\frac{1}{4nk}+\frac{1}{k},\frac{2}{k}-\frac{1}{4nk}) \cup\dots\cup [-\frac{1}{4nk}+\frac{k-1} {k},1-\frac{1}{4nk})$. Obviously, for any $r\in\mathbb R$ there is a unique number $a\in\mathbb Z$ such that $r-a\in [-\frac{1}{4nk},1-\frac{1} {4nk})$. Let $r-a\in [-\frac{1}{4nk}+\frac{j}{k},\frac{j+1}{k}-\frac{1}{4nk}) $, where $j \in\{0,\dots,k-1\}$. Since $j$ runs through the complete system of residues $\{0,1,\dots,k-1\}$ modulo $k$ and $l$ is coprime with $k$, it follows that $(-jl)$ also runs through a complete system of residues $\{0,-l,\dots,-l(k-1)\}$ modulo $k$ \cite[page 46]{Vinogradov}. Consequently, there are integers $i\in\{0,-l,\dots,-l(k-1)\}$ and $b$ such that $j+bk=i$. Then $(r-a+b)\in [-\frac{1}{4nk}+\frac{j+bk}{k},\frac{j+1+bk}{k}-\frac{1 }{4nk})= [-\frac{1}{4nk}+\frac{i}{k},\frac{1}{k}+\frac{i}{k}-\frac{1}{ 4nk}) \subset U$. Thus, $m=a-b$ is the required integer such that $(r-m)\in U$.
  
Let $\varrho(r)$ denotes an integer $\varrho(r)\in\mathbb Z$ such that $(r-\varrho(r))\in U$. Define the map $\bar X\colon S_g\times\mathbb R\to S_g\times\mathbb R$ by  the formula $\bar X(z,r)=\gamma^{-\varrho(r)}(x(\gamma^{\varrho(r)}(z,r)))$ for $(z,r) \in S_g\times\mathbb R$. Then $\bar X\gamma=\gamma \bar X$.

  {\bf Step 3.} Construct a homeomorphism ${f'}\colon M_J\to M_J$.

  Let us set $\bar {f'}=\bar X\bar f \bar X^{-1}$. Since $\bar X\gamma=\gamma \bar X$ and $\bar f\gamma=\gamma\bar f$, it follows that $\bar f'\gamma=\gamma\bar f'$ and homeomorphisms $\bar X$ and $\bar f'$ project into homeomorphisms ${f'} \colon M_{J}\to M_{J}$, $X\colon M_{J}\to M_{ J}$ (see Statement \ref{cyclic}), given by the formulas ${f'}=p_{ J}(\bar {f'}(p_{ J}^{-1}(w)))$, $X=p_{ J }(\bar X(p_{ J}^{-1}(( w)))$ and ${f'}=XfX^{-1}$.

  Let us prove that $\bar {f'}|_{S_g\times{\mathcal T}_{nk}}=\bar \varphi_{P,J,n,k,l}|_{S_g\times{\mathcal T}_{nk}}$. Since $\bar X(S_g\times\{r\})=S_g\times\{r\}$ and $\bar f(S_g\times\{r\})=S_g\times\{r+\frac {l}{k}\}$ for any $r\in\mathcal T_{nk}$, it follows that $\bar {f'}(S_g\times\{r\})=\bar X(\bar f( \bar X^{-1}(S_g\times\{r\})))=S_g\times\{r+\frac{l}{k}\}$. Then for any $r\in\mathcal T_{nk}$ the homeomorphisms ${f'}_r\colon S_g\to S_g$, $X_r\colon S_g\to S_g$ are correctly defined  by   ${f'}_r= \rho_{r+\frac{l}{k}}\bar{f'}\rho_r^{-1}$, $X_r=\rho_{r+\frac{l}{k}}\bar X\rho_r^ {-1}$ and \begin{equation}\label{psir} {f'}_r=X_{r+\frac{l}{k}}f_rX_r^{-1}. \end{equation} Then \begin{equation}\label{Xr} X_r=J^{-m(r)}x_rJ^{m(r)}. \end{equation}

By construction, $\bar \varphi_{P,J,n,k,l}(z,r)=(P(z),r+\frac{l}{k})$ and $\bar{f'}( z,r)=({f'}_r(z),r+\frac{l}{k})$ for any $r\in{\mathcal T}_{nk}$.

Let us prove that ${f'}_{r}=P$ for any $r\in{\mathcal T}_{nk}$. Let us represent $r\in{\mathcal T}_{nk}$ in the form $r=i-j\frac{l}{k}+m$, where $i\in T$, $j\in\{0,\dots,k-1\}$ and $m\in\mathbb Z$.

  Let $k=1$. Then ${f'}_r={f'}_{i+m}\overset{\eqref{psir}}{=}X_{i+m}f_{i+m}X_{i+m}^{ -1}\overset{\eqref{Xr}}{=}J^{-m}x_{i}J^{m}f_{i+m}J^{-m}x_{i}^{-1 }J^{m}\overset{\eqref{xzr}}{=}\\J^{-m}\xi_{i,0}J^{m}f_{i+m}J^{-m}\xi_ {i,0}^{-1}J^{m} \overset{\eqref{Jf0}}{=}J^{-m}\xi_{i,0}f_{i}\xi_{i,0 }^{-1}J^{m} \overset{\eqref{xi}}{=}J^{-m}h_{i}f_{i}h_{i}^{-1}J^{m }\overset{\eqref{phirSg}}{=}J^{-m}h_{i}\phi_{i}h_{i}^{-1}J^{m}  \overset{\eqref{hi}}{=}\\J^{-m}PJ^m=P$.

Let $k>1$.
We consider the cases 1) $j\geq 1$ and 2) $j=0$ separately.

1) If $j\geq 1$, then $j-1\in\{0,\dots,k-2\}$ and the homeomorphism $\xi_{i,j-1}$ is correctly defined. We obtain that ${f'}_r={f'}_{i-j\frac{l}{k}+m}\overset{\eqref{psir}}{=}X_{i-(j-1)\frac{ l}{k}+m}f_{i-j\frac{l}{k}+m}X_{i-j\frac{l}{k}+m}^{-1}\overset{\eqref{Xr}} {=}J^{-m}x_{i-(j-1)\frac{l}{k}}J^{m}f_{i-j\frac{l}{k}+m}J^{- m}x_{i-j\frac{l}{k}}^{-1}J^{m} \\\overset{\eqref{xzr}}{=}J^{-m}\xi_{i,j-1 }J^{m}f_{i-j\frac{l}{k}+m}J^{-m}\xi_{i,j}^{-1}J^{m} \overset{\eqref{Jf0}}{=}J^{-m}\xi_{i,j-1}f_{i-j\frac{l}{k}}\xi_{i,j}^{-1}J^{m} \overset{\eqref{xi}}{=}\\J^{-m}P^{-j+1}h_if_{i-(j-1)\frac{l}{k}+(j-2)\frac {l}{k}} \dots f_{i-(j-1)\frac{l}{k}}f_{i-j\frac{l}{k}}f_{i-j\frac{l}{k} }^{-1}\dots f_{i-j\frac{l}{k}+(j-1)\frac{l}{k}}^{-1}h_i^{-1}P^jJ^{ m}=\\J^{-m}P^{-j+1}h_ih_i^{-1}P^jJ^{m}=P$.

2) If $j= 0$, then $r+\frac{l}{k}=i+\frac{l}{k}+m=i-(k-1)\frac{l}{k}+( m+l)$. We obtain that ${f'}_r={f'}_{i+m}\overset{\eqref{psir}}{=}X_{i-(k-1)\frac{l}{k}+(m +l)}f_{i+m}X_{i+m}^{-1}\overset{\eqref{Xr}}{=}J^{-m-l}\xi_{i,k-1}J^ {m+l}f_{i+m}J^{-m}\xi_{i,0}^{-1}J^{m} \overset{\eqref{Jf0}}{=}J^{- m-l}\xi_{i,k-1}J^{l}f_{i}\xi_{i,0}^{-1}J^{m} \overset{\eqref{xi}}{=}\\
J^{-m-l}P^{-k+1}h_if_{i-(k-1)\frac{l}{k}+(k-2)\frac{l}{k}} \dots f_{ i-(k-1)\frac{l}{k}}J^lf_{i}h_i^{-1}J^{m}\\ \overset{\eqref{Jf0}}{=}
J^{-m-l}P^{-k+1}h_iJ^lf_{i+(k-1)\frac{l}{k}} \dots f_{i-\frac{l}{k}}f_{ i}h_i^{-1}J^{m}\overset{\eqref{phirSg}}{=}
J^{-m-l}P^{-k+1}h_i\phi_i h_i^{-1}J^{m}\\\overset{\eqref{hi}}{=}
J^{-m-l}P^{-k+1}J^lP^kJ^{m}=P$.

We obtain that   ${\bar f'}(p_J^{-1}(NW(f'))=\bar\varphi_{P,J,n,k,l}(p_J^{-1}(\varphi_{P,J,n,k,l})$. 

Consequently, ${f'}|_{NW({f'})}=\varphi_{P,J,n,k,l}|_{NW(\varphi_{P,J,n,k,l})}$ and the homeomorphism $f$ is ambiently $\Omega$-conjugate to the homeomorphism $\varphi_{P, J,n,k,l}$ via the map  $X$. \end{proof}

\section*{FUNDING} The work is supported by the Russian Science Foundation under grant 22-11-00027 except
for the results of Section 3 which was supported by the Laboratory of Dynamical Systems and Applications NRU HSE, grant of the Ministry of science and higher education of the RF, ag. № 075-15-2022-1101.

\section*{CONFLICT OF INTEREST} 
The authors declare that they have no conflicts of interest.

\end{document}